\documentclass[12pt, DIV=15]{scrartcl}
\usepackage[utf8]{inputenc}
\usepackage[english]{babel}
\usepackage{amsmath,amssymb,amsfonts,amsthm}
\usepackage{graphicx}
\usepackage{hyperref}
 \usepackage[draft]{fixme}
\newcounter{thm}
\newtheorem{theorem}[thm]{Theorem}
\newtheorem*{theorem*}{Theorem}
\newtheorem{lemma}[thm]{Lemma}
\newtheorem{proposition}[thm]{Proposition}
\newtheorem*{conjecture*}{Conjecture}

\newtheorem{remark}[thm]{Remark}
\theoremstyle{definition}
\newtheorem{definition}[thm]{Definition}

\newcommand{\eps}{\varepsilon}

\newcommand{\bbR}{\mathbb R}
\newcommand{\bbQ}{\mathbb Q}

\newcommand{\bbZ}{\mathbb Z}

\newcommand{\rot}{\mathrm{rot}\,}

\renewcommand{\mod}{\operatorname{mod}}

\title{Anti-classification for flows on two-tori}
\author{Nataliya Goncharuk}
\begin{document}
\maketitle
\begin{abstract}
We prove that the classification of real-analytic vector fields on the two-torus up to orbital topological equivalence does not admit a complete numerical invariant that is a Borel function. Moreover, smooth vector fields that are difficult to classify appear in generic smooth 7-parameter families.  In dimension 2, this improves the recent result of Gorodetski and Foreman \cite{ForGor} for non-classifiability of smooth diffeomorphisms up to continuous conjugacy. 
\end{abstract}

 \section{Introduction}
 
 Classification results  constitute one of the central parts of the modern theory of dynamical systems. For example, due to Denjoy theorem, rotation number is a complete invariant that classifies $C^2$-smooth circle diffeomorphisms $f\colon \bbR/\bbZ\to \bbR/\bbZ$ without periodic orbits up to continuous conjugacy. Separatrix skeletons, graphs, or schemes are used to classify planar vector fields up to orbital topological equivalence. Ornstein's theorem \cite{Orn}  states that entropy is a complete invariant that classifies, up to measure-preserving transformation, Bernoulli shifts on closed subsets of the space of bi-infinite sequences. 
 
On the other hand, Yoccoz's example shows that rotation number cannot be used to classify circle diffeomorphisms up to \emph{smooth} conjugacy in the case when the rotation number is Liouville.  Ph.~Kunde showed \cite{Kunde} that smooth conjugacy on the space of circle diffeomorphisms admits no complete numerical invariant that is a Borel function. This is an \emph{anti-classification} result that captures the complicated nature of the equivalence relation.

 Strong anti-classification results were obtained in ergodic theory. Consider the space  $X$ of $C^\infty$-smooth diffeomorphisms of a torus. Let the equivalence relation be a measure-preserving conjugacy. In \cite{ForWeiss}, M.~Foreman and B.~Weiss proved that this equivalence relation is not \emph{Borel}: the set $\{(S,T)\in X\times X \mid S\sim T \}$ is not Borel with respect to the $C^\infty$-topology in $X\times X$. Earlier in \cite{FGW}, M.~Foreman, D.~Rudolph, and B.~Weiss proved that measure-isomorphism for measure-preserving ergodic maps on the interval is not a Borel equivalence relation. In \cite{GeKu}, M.~Gerber and Ph.~Kunde proved that Kakutani equivalence relation for ergodic measure-preserving transformations is also not Borel. 

In Sec. \ref{sec-prelim}, we will introduce Borel reducibility, the partial order on equivalence relations that produces the hierarchy of equivalence relations (see also \cite{For18}). For many natural equivalence relations, their place in this hierarchy is not known. An important breakthrough was a paper by M.~Sabok \cite{Sab} who showed that  the isomorphism of separable $C^*$ algebras is the maximal equivalence relation among all orbit equivalence relations. J.~Zielinski \cite{Ziel} showed that the homeomorphism of compact metric spaces is also maximal among all orbit equivalence relations. It is an open question whether measure-isomorphism for measure-preserving ergodic maps has the same property.  
 
  
One of the natural equivalence relations in dynamical systems theory is continuous conjugacy.  In the space of diffeomorphisms, A.~Gorodetski and M.~Foreman \cite{ForGor} showed that this equivalence relation for smooth diffeomorphisms of $\bbR^2$ has no complete Borel numerical invariants. Moreover, for diffeomorphisms on $\bbR^5$, this equivalence relation is not Borel\footnote{ In \cite{ForGor}, authors announced stronger results, but they were not published as of 05/2025.}. However, proofs involve classification of diffeomorphisms that are highly degenerate. Related results were obtained in the space of continuous interval maps and circle maps, see \cite{BV} and references therein.

Planar vector fields can be classified up to orbital topological equivalence using a combinatorial invariant (in the form of separatrix skeletons, schemes, or Leontovich-Mayer-Fedorov graphs). Classification of vector fields on the torus is more complicated: since circle maps can appear as Poincare maps, classification invariant should incorporate both the information about the behavior of separatrices and the rotation number of the Poincare map. 
We will see that this is sufficient to obtain non-classifiability results similar to \cite[Theorem 2]{ForGor}.  

The proofs are not directly related to, but largely inspired by, new examples in the modern bifurcation theory for planar vector fields that arise from sparkling separatrix connections, see \cite{IKS}. 

Let $\mathcal V^2(T^2)$ be the space of $C^2$-smooth vector fields on the two-torus $T^2=\bbR^2/\bbZ^2$. Let $\mathcal V^\omega(T^2)$ be the space of real-analytic vector fields on the two-torus.
\begin{definition}
 Two vector fields $v,w\in \mathcal V^2(T^2)$ are orbitally topologically equivalent, $v \sim w$, if there exists a homeomorphism $H\colon T^2\to T^2$ that is homotopic to identity, such that $H$ takes orbits of $v$ to orbits of $w$, preserving time orientation.
\end{definition}
The main results of the paper are the following.
\begin{theorem}
\label{th-main}
Orbital topological equivalence in $\mathcal V^2(T^2) $ has no complete Borel numerical invariant:  
there is no Borel function 
$g \colon \mathcal V^2(T^2) \to  Y$ with $Y$ a Polish space such that for all $v, w \in X$, we have 
$v \sim w$ if and only if $g(v) = g(w)$.
\end{theorem}
\begin{theorem}
\label{th-an}
Orbital topological equivalence in $\mathcal V^{\omega}(T^2) $ has no complete Borel numerical invariant:  
there is no Borel function 
$g \colon \mathcal V^{\omega}(T^2) \to  Y$ with $Y$ a Polish space such that for all $v, w \in X$, we have 
$v \sim w$ if and only if $g(v) = g(w)$.
\end{theorem}

 By a classical Kuratowski's theorem, all uncountable Polish spaces are Borel isomorphic. In particular, any Polish space $Y$ is Borel isomorphic to $\bbR$, thus we refer to these statements as the absence of \emph{numerical} invariants for orbital topological equivalence. Results also imply that there are no complete functional invariants in any Polish functional space.   
\begin{remark}
While circle diffeomorphisms appear as first-return maps for vector fields on the torus, result of \cite{Kunde} does not imply Theorem \ref{th-main}, since we consider a different equivalence relation. 
\end{remark}
\begin{remark}
While time-1 flows of vector fields $v\in \mathcal V(T^2)$ are diffeomorphisms of $T^2$, continuous conjugacy for resulting diffeomorphisms does not coincide as equivalence relation to orbital topological equivalence of corresponding vector fields. So Theorem \ref{th-main} does not imply  \cite[Theorem 2]{ForGor}. However, equivalence of vector fields is considered to be much simpler than equivalence of planar diffeomorphisms (e.g. Newhouse phenomenon does not happen for flows of vector fields). So in a sense, our result is stronger than  \cite[Theorem 2]{ForGor}. Also, methods of \cite{ForGor} do not allow analytic diffeomorphisms, in contrast with Theorem \ref{th-an}.  It is worth mentioning that the main result of \cite{ForGor} (Theorem 1) is a much stronger anti-classification result for diffeomorphisms of $\bbR^5$. 
\end{remark}

 Below we will formulate and prove a stronger version of Theorem \ref{th-main}: vector fields that are difficult to classify appear in a generic 7-parameter family, see Theorem \ref{th-gen}.
 The following questions are open.

\emph{ Can we find a two-dimensional manifold $M$ such that the orbital topological equivalence of $C^\omega$ vector fields on $M$ is not Borel?}
  
\emph{ Can we find a generic finite-parameter family of vector fields $v_\rho$, $\rho\in \bbR^k$ on a two-dimensional manifold such that the orbital topological equivalence is not Borel: the graph $\{(\rho_1, \rho_2)\in \bbR^{2k} \mid v_{\rho_1}\sim v_{\rho_2} \}$ of the orbital topological equivalence relation on this family is not a Borel set?  }
 
 We refer the reader to \cite{open} for the list of open questions in descriptive set theory related to dynamical systems.
 
 \section{Preliminaries: Borel reduction}
 \label{sec-prelim}
 We refer the reader to \cite{For18} for an expository introduction to Borel reduction and the hierarchy of equivalence relations with respect to the Borel reduction.

Recall that a topological space is called \emph{Polish} if it is separable and completely metrizable (i.e. admits a complete metric that is compatible with the topology). The $\sigma$-algebra of \emph{Borel sets} of $X$
 is the smallest $\sigma$-algebra containing all open sets. A map $f\colon X\to Y$
 between two topological spaces is called \emph{Borel}  if  for any open set $A$, $f^{-1}(A)$ is a Borel set. 

Consider an  equivalence relation on the set $X$. We write $x\sim_E y$ if $x,y$ are equivalent with respect to $E$. 
\begin{definition}
An equivalence relation $E$ on $X$ is \emph{smooth} if 
 there exists a Polish
space $Y$ and a Borel function $f \colon X \to Y$ such that $x \sim_E y$ holds if and only if $f(x) = f(y)$ for all $x, y\in X$.
\end{definition}
Theorem \ref{th-main} means that orbital topological equivalence of vector fields is non-smooth. 

We will use the following non-smooth equivalence relation. 
\begin{definition}
Let $E_\alpha$ be the equivalence relation on the circle $\bbR/\bbZ$ given by $x\sim_{E_{\alpha}} y$ if $x=y +n\alpha \mod 1$. 
\end{definition}
The following proposition is the particular case of the general result (see \cite[Theorem 1.1]{HKL}): an equivalence relation that admits a non-atomic ergodic measure is not smooth. This statement is a part of  the Glimm-Effros dichotomy (first discovered in \cite{G}, \cite{E} for group actions). For completeness, we will give an elementary proof of the proposition below.  
\begin{proposition}
\label{prop-Ephi}
For any $\alpha\notin \bbQ$, for any nonempty open interval $I\subset \bbR/\bbZ$,  $E_\alpha$ is a non-smooth equivalence relation both on  $I$ and on $I\setminus \{n\phi\}_{n\in \bbZ}$.
\end{proposition}
\begin{proof}
Let $R_\alpha(x)=x+\alpha$ be a rotation on a circle $\bbR/\bbZ$. 
If there is a Borel numerical invariant $f\colon I\to \bbR$  for $E_\alpha$, then the sets $A_y =f^{-1}((-\infty,y))\subset I$ must be Borel, and thus Lebesgue measurable. Since every set $A_y$ is an intersection of $R_\alpha$-invariant measurable set with $I$, and $R_\alpha$ is ergodic, the measure of each set $A_y$ is equal to $0$ or to $\mu(I)$.  Let $x$ be a supremum of the set $\{y\in \bbR, \mu (A_y)=\mu(I)\}$.  Since $\mu(\cap A_n)=\mu(\varnothing)=0$ and $\mu(\bigcup A_n)=\mu(I)$, $x$ is a finite real number. Then $\mu (f^{-1}((-\infty, x+1/n))=\mu(I)$ for any $n$ and thus $\mu (f^{-1}((-\infty, x])=\mu(I)$. On the other hand,  $\mu (f^{-1}((-\infty, x-1/n))=0$ for any $n$ and thus $\mu (f^{-1}((-\infty, x))=0$. We conclude that $f^{-1}(x) $ has measure $\mu(I)$. This is impossible since $f^{-1}(x)$ is the intersection of a single orbit of an irrational rotation with $I$ and has measure 0. The proof for $I\setminus \{n\phi\}_{n\in \bbZ}$ is analogous. 
\end{proof}
Our main tool is a Borel reduction for equivalence relations. 

\begin{definition}
Let $E\subset X\times X$ and $F\subset Y\times Y$ be equivalence relations on Polish spaces $X$ and $Y$, respectively. A Borel function $f \colon  X \to Y$ is called a \emph{Borel reduction} of $E$ to $F$ if for all $x_1$, $x_2 \in X$, we have that  $x_1 \sim_E x_2$ if and only if $f(x_1)\sim_{F} f(x_2)$.

We say that $E$ is Borel reducible to $F$, and write $E \le _B F$.
\end{definition}
Informally, if $E$ is Borel reducible to $F$, then $F$ is ``not less complicated'' than $E$. In particular, an equivalence relation $E$ is smooth if and only if $E \le_B =_{\bbR}$, where $=_{\bbR}$ is the equality relation on $\bbR$. 

Borel reducibility is a partial order; hence  if $E_1\le_B E_2$ and $E_1$ is non-smooth, then $E_2$ is also non-smooth (otherwise $E_1 \le_B E_2 \le =_{\bbR}$). We conclude that Theorem \ref{th-main} is implied by the following. 
 \begin{theorem}
 \label{th-family}
 For each irrational number $\phi$, there exists a smooth family of vector fields $v_{\rho,\phi} \subset \mathcal V^2(T^2)$, $\rho\in \bbR/\bbZ$,  such that the equivalence relation $E_\phi$ on $\bbR\setminus \{n\phi\}_{n\in \bbZ}$ is Borel reducible to orbital topological equivalence: for $\rho_1, \rho_2\in \bbR\setminus \{n\phi\}_{n\in \bbZ}$, we have $v_{\rho_1, \phi}\sim v_{\rho_2, \phi}$ if and only if $\rho_1  = \rho_2 \mod n\phi$. 
 \end{theorem}


%
%
 
 This theorem is proved in the next three sections.
 
\section{Construction of the family  $v_{\rho,\phi}$.}

Recall that we consider vector fields on the two-torus $T^2=(\bbR/\bbZ)^2$.
Let  $M_s= \bbR/\bbZ\times \{s\}$ be the meridians of the torus. 
 Fix $\phi\in [0,1]\setminus \bbQ$ and $\rho\in \bbR/\bbZ$.

\begin{figure}[h]
\begin{center}
\includegraphics[width=0.4\textwidth]{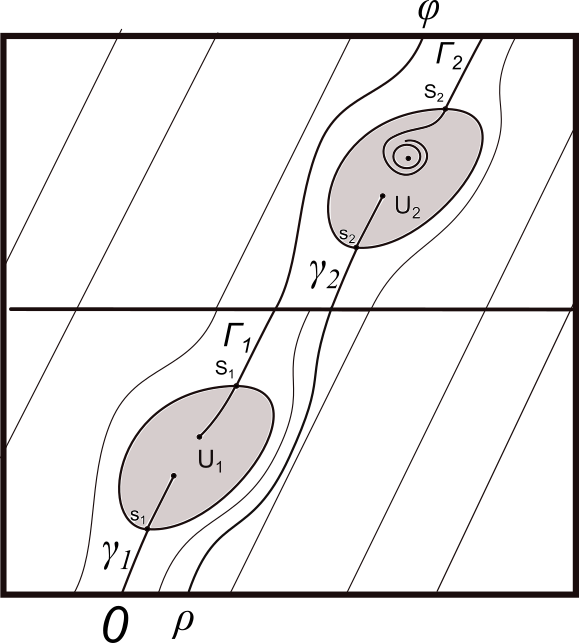}
\caption{Vector field $v_{\rho, \phi}$. }
\end{center}
\end{figure}
   
We define $v_{\rho, \phi}$ in the following way. 
  \begin{itemize}

\item On neighborhoods of $M_0$ and $M_{0.5}$, define $v_{\rho,\phi}=(\phi, 1)$.

\item Let the vector field $v_{\rho,\phi}$  in $\{0<y<0.5\}$ have two saddles $s_1, S_1$. 

One stable separatrix $\gamma_1$ of $s_1$ intersects $M_0$ at $(0,0)$, two unstable separatrices of $s_1$ form separatrix connections with two stable separatrices of $S_1$, an unstable separatrix $\Gamma_1$ of $S_1$ intersects $M_{0.5}$ at $(\phi/2, 1/2)$. Let $U_1$ be bounded by the separatrix connections; in $U_1$, the vector field $v_{r,\phi}$ has one stable and one unstable node. 

The correspondence map along $v_{r,\phi}$ between the sides of  $R_1$ is given by $(x,0)\to (x+\phi/2,1/2)$,  except it is undefined at $(0,0)$. 

\item For $\rho=0$, the vector field $v_{0,\phi}$ in the domain $\{0.5<y<1\}$ has a similar structure, with saddles $s_2, S_2$ and separatrices $\gamma_2,\Gamma_2$ that intersect $M_{0.5}$ and $M_1$ at $(\phi/2, 1/2)$ and $(\phi, 1)$ respectively.  The only difference is, that in the domain $U_2$ bounded by its separatrix connections, the vector field $v_{\rho,\phi}$ has one stable node and one unstable limit cycle with a stable node in it. 

\item For other values of $\rho$, in the domain $\{0.5<y<1\}$, we put $v_{\rho, \phi}(x,y) = v_{0,\phi} (x-\rho, y). $ 
 \end{itemize}

With this definition, $\rho\to v_{\rho,\phi}$ is a smooth family of vector fields for any fixed $\phi$.  
The Poincare map under $v_{\rho,\phi}$ from the meridian $M_0$ of the torus to itself coincides with $x\to x+\phi$ everywhere, except it is undefined at the intersections with $\gamma_1, \gamma_2$. 
If  $\rho\notin \{n\phi\}_{n\in \bbZ}$, then $\Gamma_2$ does not coincide with $\gamma_1$ and $\Gamma_1$ does not coincide with $\gamma_2$. 
In this case, we have  $M_0\cap \gamma_1 =\{-n\phi\}_{n=0}^{\infty} $ and $M_0\cap \gamma_2=\{\rho-n\phi\}_{n=0}^{\infty}$. 

Recall that if a point $a$ is non-singular for a vector field, $v(a)\neq 0$, then there exists a  neighborhood $U$ with a smooth chart $H\colon U\to \bbR^2$ such that $H_* v = (1,0)$. 
A continuous curve $\gamma$ is \emph{topologically transverse} to the vector field if it does not pass through singular points, and for any point $a\in \gamma$ there exists its neighborhood $U$ such that the image $H(\gamma\cap U)$ is a graph of a continuous function $x=x(y)$. The next lemma follows from elementary properties of correspondence maps. 
\begin{lemma}
\label{lem-homeo}
For any simple closed loop $\alpha\subset T^2$ homotopic to the meridian $M_0$ that does not intersect $\overline U_1\cup \overline U_2$ and is topologically transverse to $v_{\rho, \phi}$, there exists a homeomorphism  $\xi\colon \alpha\to M_0$ with the following property: if  $\xi(p_1)=p_2$, then either $p_1, p_2$ belong to the same trajectory of $v_{\rho, \phi}$, or $p_1\in \Gamma_{1}\cup \Gamma_2$ and $p_2\in \gamma_{1}\cup \gamma_2$. 
\end{lemma}
\begin{proof}
Lift $v_{\rho, \phi}$ to the vector field $\hat v$ on the cylinder $\bbR/\bbZ\times \bbR$.
Lift $\alpha$ and $M_0$ to the cylinder $\bbR/\bbZ\times \bbR$ so that the lifts $\hat \alpha$, $\hat M_0=\{y=0\}$ do not intersect and $\hat \alpha$ is above $\hat M_0$. Then the correspondence map $\hat \xi\colon \hat \alpha \to \hat M_0$ along trajectories of $-\hat v$ is well-defined. Indeed, since both curves are topologically transverse to $\hat v$, the only obstructions for extending the correspondence  map are intersections of $\hat M_0$ with stable separatrices  of $\hat v$ and intersections of $\hat \alpha$ with unstable separatrices of $\hat v$. Since $\hat M_0$ and $\hat \alpha$ do not intersect $\overline U_1, \overline U_2$, these are the intersections of the lifts of  $\gamma_{1,2}$ with $\hat M_0$, and of the lifts of $\Gamma_{1,2}$ with $\hat \alpha$. Correspondence map $\hat \xi$ extends continuously to these intersections. It descends to the map  $\xi\colon \alpha\to  M_0$ that  satisfies assumptions of the lemma. 
\end{proof}

\section{Equivalent vector fields have $E_\phi$-equivalent parameters}

\begin{lemma}
\label{lem-eqvf}
For irrational $\phi$, let $v_{\rho, \phi}$ be vector fields constructed above.
 Suppose that $\rho_1,\rho_2\notin \{n\phi\}_{n\in \bbZ}$.
  
If vector fields $v_{\rho_1,\phi_1},v_{\rho_2, \phi_2}$ are orbitally topologically equivalent, then $\phi_1=\phi_2 $ and $\rho_1 = \rho_2 +n\phi \mod 1$. 
\end{lemma}
\begin{proof}
Let $s_{1,2},S_{1,2}, \gamma_{1,2}, \Gamma_{1,2}, U_{1,2}$ be as defined above for $v_{\rho_1, \phi_1}$, and let $\tilde s_{1,2}, \tilde S_{1,2}, \tilde \gamma_{1,2}, \tilde \Gamma_{1,2}, \tilde U_{1,2}$ be analogous objects for $v_{\rho_2, \phi_2}$. 

Suppose that $H$ is an orbital topological equivalence between $v_{\rho_1,\phi_1}$ and $v_{\rho_2,\phi_2}$. 
Since $H$ takes attractors and repellors of $v_{\rho_1}$ to attractors and repellors of $v_{\rho_2}$ respectively, and limit cycles to limit cycles, we have $H(U_{1})=\tilde U_1$ and $H(U_2)=\tilde U_2$; $H(s_{1,2})=\tilde s_{1,2}$ and $H(S_{1,2})=\tilde S_{1,2}$; therefore  $H({\gamma_{1}})=\tilde \gamma_{1}$ and $H({\gamma_{2}})=\tilde \gamma_{2}$. (Here we used that phase portraits of $v_{\rho, \phi}|_{U_1}$ and $v_{\rho, \phi}|_{U_2}$ are different, otherwise $H$ could map $\gamma_1$ to $\tilde \gamma_2$ and $\gamma_2$ to $\tilde \gamma_1$.)

The curve  $H(M_0)$ is topologically transverse to $v_{\rho_2}$, does not intersect $\tilde U_1\cup \tilde U_2$ and is homotopic to $M_0$, since $H$ is homotopic to identity.  Using Lemma \ref{lem-homeo}, define a homeomorphism $\xi\colon H(M_0)\to M_0$ along trajectories of $v_{\rho_2, \phi_2}$. We get an orientation-preserving circle homeomorphism $\xi\circ H\colon M_0 \to M_0$.
  
Recall that the Poincare maps on $M_0$ under the action of  $v_{\rho_1,\phi_1}$ and $v_{\rho_2, \phi_2}$ equal $x\to x+\phi_1$, $x\to x+\phi_2$ respectively. 
 Since $\xi\circ H$ conjugates these Poincare maps, we have $\phi_1=\phi_2$. From now on, we will omit $\phi_1, \phi_2$ from notation. 
 
Consider the set $A_{1} = (M_0\cap \gamma_1 )=\{-n\phi\}_{n\in \bbZ}\times \{0\}$. The set  $H(A_1)$ belongs to $H(M_0)\cap \tilde \gamma_1$. Since $\rho_2\notin \{n\phi\}_{n\in \bbZ}$, the separatrix $\tilde \gamma_1 $ does not form a separatrix connection with  $\tilde \Gamma_{1,2}$; Lemma \ref{lem-homeo} implies that the set  $\xi(H(A_1))\subset M_0$ belongs to $\tilde \gamma_1$ as well.  

Hence the circle homeomorphism $\xi\circ H$ takes the dense set $A_1=\{-n\phi\}_{n\in \bbZ}\times \{0\}$ into a subset of $\{-n\phi\}_{n\in \bbZ}\times \{0\}$. We conclude that $\xi\circ H$ must be a rotation by $ k\phi$, $k\in \bbZ$. 

On the other hand, analogous arguments for $\gamma_2$ imply that $\xi\circ H$ takes $M_0\cap \gamma_2 =\{\rho_1-n\phi\}_{n\in \bbZ}\times \{0\}$ into a subset of $M_0\cap \tilde \gamma_2 = \{\rho_2-n\phi\}_{n\in \bbZ}\times \{0\}$, hence $\xi \circ H$ must be a rotation by $\rho_2 -\rho_1+ l\phi $, $l\in \bbZ$. We conclude that $ \rho_2-\rho_1 = m\phi \mod 1$.   
\end{proof}

\section{Vector fields with $E_\phi$-equivalent parameters are equivalent}
\begin{lemma}
\label{lem-eqphi}
For irrational $\phi$, let $v_{\rho, \phi}$ be the vector field constructed above.
 Suppose that $\rho_1,\rho_2\notin \{n\phi\}_{n\in \bbZ}$, and $\rho_1 = \rho_2 +n\phi\mod 1$. Then $v_{\rho_1, \phi}$ and $v_{\rho_2, \phi}$ are orbitally topologically equivalent. 
\end{lemma}
\begin{proof}
We will write $v_{\rho}$ instead of $v_{\rho, \phi}$ for brevity. 
 Let $\eps$ be small so that the intervals $I_k=[\rho_2+k\phi-\eps, \rho_2+k\phi+\eps]=[\rho_1-(n-k)\phi-\eps, \rho_1-(n-k)\phi+\eps]$ do not intersect for $k=0,1,\dots,n$ and do not intersect $J=[-\eps, \eps]$.

Let $v_{\rho_1}$ be defined on the torus $T$, and $v_{\rho_2}$ be defined on the torus $\tilde T$.  We split $T$ into a union  of the following three connected sets: 
  \begin{itemize}
\item  (1) $V_1$. Let $V_1^1$ be the union of  arcs of trajectories of $v_{\rho_1}$ that start at $(x,0), x\in [-\eps, \eps] =J$, and end at $(x+\phi,0)$. Let $V_1 = \overline U_1 \cup \overline {V_1^{1}}$.

\item (2) $V_2$. Let $V_2^1$ be the union of  arcs of trajectories of $v_{\rho_1}$ that start at $(x,0), x\in I_0$, and end at $(x+n\phi,0)\in I_{n}$. Let $V_2 =  \overline U_2 \cup \overline {V_2^{1}}$.
  
  \item (3) $V_3 = T\setminus V_1\setminus V_2$.  
  \end{itemize}
  The set  $V_1$ contains the arc of $\gamma_1$ before its first intersection with $M_0$ that belongs to $J$. The set $V_2$ contains the arc of $\gamma_2$ and its 1st, 2nd, $\dots$, $n$th intersections with $M_0$. Intersections happen inside $I_{n-1}$, $\dots$, $I_0$ respectively.
   
Similarly, we define sets $\tilde V_1, \tilde V_2, \tilde V_3$ for $v_{\rho_2}$, using the same intervals $I_0,\dots I_n,J$ for the field $v_{\rho_2}$. Again, the set  $\tilde V_1$ contains the arc of $\tilde \gamma_1$ before its first intersection with $M_0$ that belongs to $J$. The set $\tilde V_2$ contains the arc of $\gamma_2$ before its first intersection with $M_0$ that belongs to $I_0$, and the arc of $\Gamma_2$ that contains its 1st, 2nd, $\dots$, $(n-1)$-th intersections with $M_0$. These intersections happen inside $I_1$, $I_2$, $\dots$, $I_n$ respectively.
   Fig. \ref{fig-equiv} shows domains $V_1$ -- $V_3$ and $\tilde V_1$ -- $\tilde V_3$.
\begin{figure}
\includegraphics[width=0.7\textwidth]{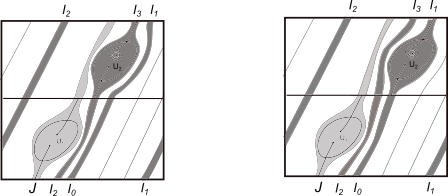}
\caption{Domains $V_1$ (light-gray), $V_2$ (dark-gray) and $V_3$ (white) for vector fields $v_{\rho_1}$ (left) and $v_{\rho_2}$ (right) for $n=3$. }\label{fig-equiv}
\end{figure}
Now, construct $H$.

Clearly,  $v_{\rho_1}|_{V_1}$ is orbitally topologically equivalent to $v_{\rho_2}|_{\tilde V_1}$. 
We will choose equivalence $H$ that is identical  on $J\times \{0\}$ and $(J+\phi)\times \{1\}$. 

Vector fields $v_{\rho_1}|_{V_2}$ and $v_{\rho_2}|_{\tilde V_2}$ are also orbitally topologically equivalent.   We will choose $H$ that is identical on the bottom and top sides $ I_0\subset M_0$,  $I_n \times \{1\}$ of $V_2$. 

Finally, $V_3\setminus M_0$ is a union of strips where $v_{\rho_1}$ is orbitally topologically equivalent to the unit vector field; the same holds for $v_{\rho_2}$ in $\tilde V_3\setminus M_0$. Thus we can extend $H$ to $V_3$ by setting $H$ to be identity on $M_0\cap \overline V_3= M_0\setminus (J\cup I_1\cup I_2\cup \dots I_{n-1})$ and extending it along trajectories of $v_{\rho_1},v_{\rho_2}$. This completes the construction of the orbital topological equivalence.  
\end{proof}

Lemmas \ref{lem-eqvf} and \ref{lem-eqphi} imply Theorem \ref{th-family}. Due to Sec. \ref{sec-prelim}, this completes the proof of Theorem \ref{th-main}. 

\section{Genericity of nonclassifiable vector fields}
In this section, we prove a stronger version of Theorem \ref{th-main}: vector fields that are orbitally topologically equivalent to $v_{\rho, \phi}$ form (at least) a codimension-7 submanifold $\mathcal M$ in the space $\mathcal V^2(T^2)$ of smooth vector fields.  This implies that they appear in generic smooth 7-parameter families of vector fields. 
\begin{theorem}
\label{th-gen}
For any irrational $\phi$, there exists a codimension-7 continuous submanifold $\mathcal M\subset \mathcal V^2(T^2)$  such that any vector field $v\in \mathcal M$ is orbitally topologically equivalent to some vector field of the form $v_{\rho,\phi}$ described in Theorem \ref{th-main}. 
\end{theorem}

To define the set $\mathcal M$, we will need the notion of the rotation number. Let $f\colon \bbR/\bbZ\to \bbR/\bbZ$ be a circle homeomorphism, and let $F\colon \bbR\to \bbR$ be its lift to the real line. The rotation number of the circle homeomorphism $f$ is given by $$\rot f = \lim_{n\to \infty}\frac{F^n(x)}{n}.$$
The rotation number is rational if and only if $f$ has a periodic orbit and depends continuously on $f$. In a 1-parameter family $f_t$, if $\frac {d}{dt} f_t>0$, then the rotation number is monotonic with respect to $t$, and strictly monotonic whenever $\rot (f_t)$ is irrational.

We will also need the notion of the characteristic number of a saddle. 
Recall that if $\lambda_1<0<\lambda_2$ are the eigenvalues of the linearization matrix of a vector field at a saddle singular point, then $ |\lambda_1|:|\lambda_2|$ is called the characteristic number of a saddle. It is invariant under smooth changes of space and and time variables.  If $L_1, L_2$ are transversals to the stable and unstable separatrices of a saddle of a vector field $v$, then the correspondence map along $v$ from $L_1$ to $L_2$  is defined on semi-transversals; let these transversals be given by $\{x>0\}$ and $\{y>0\}$ in local coordinates on $L_1, L_2$. This correspondence map is called the  Dulac map. The following lemma is known to specialists. 
\begin{lemma}
\label{lem-Du}
Dulac map for the saddle with characteristic number $\mu$ has the form $x\to cx^{\mu}(1+o(1))$ on a neighborhood of zero where $c$ is a nonzero constant. 
\end{lemma}
\begin{proof}
The proof repeats the computation in the proof of Lemma 5 in \cite{IKS}.  Change variables  so that saddle separatrices become coordinate axes, and $L_1, L_2$ become $\{y=1\} $ and $\{x=1\}$ respectively.  Differential equation takes the form $x'=xg_1(x,y), y'=yg_2(x,y)$. After time change, equation becomes $x'=x, y'=yg(x,y)$, with $g$ smooth, $g(0,0)=-\mu$. Rescaling $x\to x/\eps$, $y\to y/\eps$ in an $\eps$-neighborhood of zero brings the equation to the form $ x'=x,  y'= y\tilde g( x, y)$ with $|\tilde g( x, y)-g(0,0)|<O(\eps)|x|+O(\eps)|y|$. Trajectory of the new vector field that starts at $(x_0, 1)$ has the form $x(t)=x_0e^t$, $\log y(t) = \int_0^t \tilde g(x,y )dt$. Hence it takes the time $T=-\log x_0$ for this trajectory to land on the transversal $(1,y)$. An estimate on $\tilde g(x,y)$ above implies that $ y(T)=C(x_0)e^{-\mu T}  = C(x_0)\cdot x_0^\mu$ with $1-O(\eps)<C(x)<1+O(\eps)$. Thus the Dulac map along the initial vector field $v$ between $ \{y=\eps\}$ and $\{x=\eps\}$ has the form $y=c(x)\cdot x^\mu$ with $1-O(\eps)<c(x)/c_0<1+O(\eps)$ for certain $c_0$. Since the correspondence map between $L_1=\{y=1\}$ and $\{y=\eps\}$ is smooth, as well as the correspondence map between $L_2=\{x=1\}$ and $\{x=\eps\}$, we conclude that on a sufficiently small neighborhood of zero on $L_1$,  the Dulac map from $L_1$ and $L_2$ has the form $y=d(x)\cdot x^\mu$ with $1-O(\eps)<d(x)/d_0<1+O(\eps)$ for certain $d_0$. Since $\eps$ was arbitrary, this implies the statement.  
\end{proof}

\begin{proof}[Proof of Theorem \ref{th-gen}]
\textbf{Construction of $\mathcal M$.} 

Take a vector field $v_{\rho, \phi}$ with small $\rho$. Modify it is needed to guarantee that on a neighborhood of $\{x=0\}$, we have  $v_{\rho,\phi}=(1,\phi)$. Consider its small neighborhood $\mathcal U$ in the space $\mathcal V^2(T^2)$ of $C^2$-smooth vector fields in $T^2$. Let $s_1(v), s_2(v), S_1(v), S_2(v)$ be saddles of $v$, $v\in \mathcal U$, that are close to $s_1, s_2,  S_1, S_2$. Let $l_1$ be an interval transverse to the left separatrix connection of $s_1$ and $S_1$ of $v_{\rho, \phi}$; let  $\alpha(v)$ and $\beta(v)$ be first intersections of separatrices of $s_1(v), S_1(v)$ with $l_1$, in a local chart on $l_1$. Define $\delta_1(v)= \alpha(v)-\beta(v)$, which is a smooth function of $v$. In a similar way, define functions $\delta_2(v),\delta_3(v), \delta_4(v)$ for each of the separatrix connections of $v$. 
Note that we have $\delta_k(v_{\rho, \phi})=0$, $k=1, 2, 3, 4$, since $v_{\rho, \phi}$ has four separatrix connections. 
Define $$\mathcal M^0=\{v\in\mathcal V^2(T^2)\mid \delta_k(v)=0, k=1,2,3,4\}.$$

Denote the characteristic numbers of the saddles $s_1(v),s_2(v), S_1(v), S_2(v)$ by $\mu_1(v), \mu_2(v), \nu_1(v), \nu_2(v)$ respectively.   

For any $v\in \mathcal M^0$, let $P_v$ be the Poincare map along $v$ from $M_0$ to itself. Formally, $P_v$ is undefined at the first intersections of  separatrices $\gamma_1(v)$, $\gamma_2(v)$ with $M_0$, but it extends continuously to these points. Denote these points $A(v), B(v)$. 
On the left semi-neighborhood of $A(v)$, the map $P_v$ is a composition of two Dulac maps, from $M_0$ to $l_1$ and from $l_1$ to $M_0$. Similarly, $P_v$ is a composition of two Dulac maps on the right semi-neighborhood of $A(v)$.  Due to Lemma \ref{lem-Du},  $P_v$ has the following form: $$
\begin{aligned}x \to P_v(A(v))+ C_1(v) \cdot (x-A(v))^{\mu_1(v)\cdot \nu_1(v)}(1+o(1))\text{ for } x<A(v)\\
x \to P_v(A(v))+ C_2(v) \cdot (x-A(v))^{\mu_1(v)\cdot \nu_1(v)}(1+o(1)) \text{ for }x>A(v)\end{aligned}.$$
Note that constants $C_1(v)$ and $C_2(v)$ do not necessarily coincide. 
Since the Poincare map from $M_0$ to itself along $v_{\rho, \phi}$ is identity, we get that $\mu_1(v_{\rho, \phi})\cdot \nu_1(v_{\rho, \phi})=1$; similarly, $\mu_2(v_{\rho, \phi})\cdot \nu_2(v_{\rho, \phi})=1$.

Let  $\mathcal M^1\subset \mathcal U$ be given by  $$\mathcal M^1= \{v\in \mathcal M^0 \mid \mu_1(v)\cdot \nu_1(v)=1, \mu_2
(v)\cdot \nu_2(v)=1\}.$$ Then $\mathcal M^1$ is  a codimension-6 smooth submanifold in $\mathcal V^2(T^2)$. 
Condition on characteristic numbers implies that $P_v$ has nonzero one-sided derivatives on both sides of $A(v), B(v)$ for all $v\in \mathcal M^1$. So  $P_v$ is a $C^2$-smooth circle homeomorphism with two break points $A(v), B(v)$. 

Finally, the set $\mathcal M$ is given by  $$\mathcal M= \{v\in \mathcal M^1, \rot (P_v)=\phi\}.$$ 

\textbf{Codimension of $\mathcal M$.}

We will prove that $\mathcal M$ is a codimension-1 continuous submanifold  (i.e. the graph of a continuous function) in $\mathcal M^1$. Indeed, let $R_t\colon \bbR^2\to \bbR^2$ be the counterclockwise rotation by the angle $t$. Let $v=(v^1, v^2)$, and let $$\mathcal L=\{v\in \mathcal M^1 \mid \frac{v^1(0)}{v^2(0)} = \frac{v^1_{\rho, \phi }(0)}{v^2_{\rho,\phi}(0)}\}.$$ Since on a neighborhood of a boundary of the unit square,  we have $v_{\rho, \phi}=(1,\phi)$, the vector field $R_t^* v$ is well-defined on the torus for small $t$. Moreover, $v\in \mathcal M^1$ implies $R_t^* v \in \mathcal M^1$ for small $t$, therefore the set $\mathcal M^1$ can be locally represented as a Cartesian product $\bbR\times \mathcal L$ in the smooth chart $(t, v)\to R_t^* v (x,y)$. 

We have $\frac{d}{dt} P_{R_t^* v_{\rho,\phi}}>0$ due to the construction of $v_{\rho,\phi}$, thus $\frac{d}{dt} P_{R_t^* v}>0$ for $v$ close to $v_{\rho,\phi}$. Properties of the rotation number imply that for any $v$, the set  $\mathcal M$ intersects each fiber $R_t^* v$ on a single point, i.e. $\mathcal M$ is a graph in $\mathcal M^1$.  Since $\rot(\cdot)$ is continuous, $\mathcal M$ is the graph of the continuous function. Hence $\mathcal M$ is a continuous manifold of codimension $7$. 

\textbf{Finding orbitally topologically equivalent $v_{R,\phi}$.} 

For any $v\in \mathcal M$, we will find $R$ such that $v$ is orbitally topologically conjugate to $v_{R,\phi}$.
It is an easy generalization of a classical Denjoy theorem that a circle homeomorphism with breaks that has irrational rotation number is continuously conjugate to the rotation $x\to x+\rot(f)$ (see e.g. \cite[Theorem 2.4]{DzLiMa}). Let $\xi\colon \bbR/\bbZ\to \bbR/\bbZ$ conjugate $P_v$ to the rotation $x\to x+\phi$. Post-composing $\xi$ with rotation, we may and will assume that $\xi(A(v))=0$. Let $$R=\xi(B(v)).$$

\textbf{Constructing conjugacy.}

Now, we will prove that $v\in \mathcal M$ is orbitally topologically conjugate to $v_{R, \phi}$. Let $\gamma_1^R$, $\gamma_2^R$ denote separatrices of the vector field $v_{R, \phi}$. 

 Lift $v,v_{R,\phi}$ to vector fields $\tilde v, \tilde v_{R,\phi}$ in closed cylinders $C=\bbR/\bbZ\times [0,1]$. Comparing phase portraits, we can see that  $\tilde v$ is orbitally topologically equivalent to $\tilde v_{R,\phi}$ on $C$. Choose topological equivalence  $H$ to coincide with $\xi$ on the lower boundary of $C$, $H(x, 0) = (\xi(x),0)$. This is possible since $\xi$ takes $A(v)$ to $0$ and $B(v)$ to $R$, i.e. matches intersection points of separatrices of $\tilde v$ with $M_0$ to the  intersection points of separatrices of $\tilde v_{R,\phi}$ with $M_0$.
  
   Since $H$ maps trajectories of $\tilde v$ to trajectories of $\tilde v_{R, \phi}$, we get  $H(P_v(x),1) = (\xi(x)+\phi, 1)$ on the upper boundary of $C$. Hence  $H(x,1) = (\xi(P_v^{-1}(x))+\phi, 1) = (\xi(x)-\phi+\phi, 1) = (\xi(x),1)$. Since $H(x,0) = (\xi(x),0)$ and $H(x,1)= (\xi(x),1)$, the map   $H$ descends to a continuous map on $T^2$. This completes the proof.

\end{proof}
\begin{remark}
We could simplify the family $v_{\rho,\phi}$ by replacing two saddles in $U_1$ and/or $U_2$ with a single saddle that has a separatrix loop (cf. the next section), thus improving the codimension. However, in this case, the Poincare map will be necessarily critical with non-symmetric critical points. Such circle maps are not well-studied, and we could not find a reference to the analogue of the Denjoy theorem that applies to this case.  
\end{remark}

\section{Analytic vector fields}

In this section, we prove Theorem \ref{th-an}. We provide an explicit analytic family with no complete numerical invariants;  the proof is computer-assisted. 

Consider the family of Hamiltonian vector fields $$v_{\phi, b, c, d}= (\frac{d}{dy} u_{\phi, b, c, d}, -\frac{d}{dx} u_{\phi, b, c, d})
$$
 with the Hamiltonian 
\begin{equation}\label{eq-H}
u_{\phi, b, c, d}(x,y)=x-\phi y+(\cos y-1)(b\sin(x-y)+c\sin(x)+d\cos(y)).
\end{equation}
We will prove the following theorem; it implies Theorem \ref{th-an} due to Proposition \ref{prop-Ephi}. 
\begin{theorem}
\label{th-analytic}
For some open interval  $I\subset \bbR/\bbZ$, the equivalence relation $E_\phi$ on $I\setminus \{n\phi\}_{n\in \bbZ}$ is Borel reducible to the orbital topological equivalence on a subset of the family  $v_{\phi, b, c, d}$. 

Namely, there exist analytic functions $D$ and $\rho$ defined on an open set $V$ in $\bbR^3$, such that  the function $\rho$ is non-constant on $c$ for fixed $b,\phi$, and two vector fields $v_1=v_{\phi_1, b_1, c_1,D(\phi_1, b_1, c_1)}$ and $v_2=v_{\phi_2, b_2, c_2, D(\phi_2, b_2, c_2)}$ for $(\phi_{1,2}, b_{1,2}, c_{1,2})\in U$,   irrational $\phi_1, \phi_2$, and   $\rho(\phi_1, b_1, c_1)\notin \{n\phi_1\}_{n\in \bbZ}$, $\rho(\phi, b_2, c_2) \notin \{n\phi_2\}_{n\in \bbZ}$ are orbitally topologically equivalent if and only if $\phi_1=\phi_2=\phi$ and $\rho(\phi, b_1, c_1) = \rho(\phi, b_2, c_2) \mod n\phi$ in $\bbR/\bbZ$. 
\end{theorem}
\begin{proof}
First, let us explain how the second part of the theorem implies its first part. Define an analytic function $C(\phi, r)$ implicitly on some open set by a condition $\rho(\phi, b_0, C(\phi, r))=r$. Then for fixed irrational $\phi$, the Borel reduction of  $E_\phi$ on $I\setminus \{n\phi\}_{n\in \bbZ}$ to the orbital topological equivalence is given by $\rho\to v_{\phi, b_0, C(\phi, \rho), D(\phi, b_0, C(\phi, \rho))}$.  This implies the first part of  Theorem \ref{th-an}.

\begin{figure}[h]
\begin{center}
\includegraphics[width=0.6\textwidth]{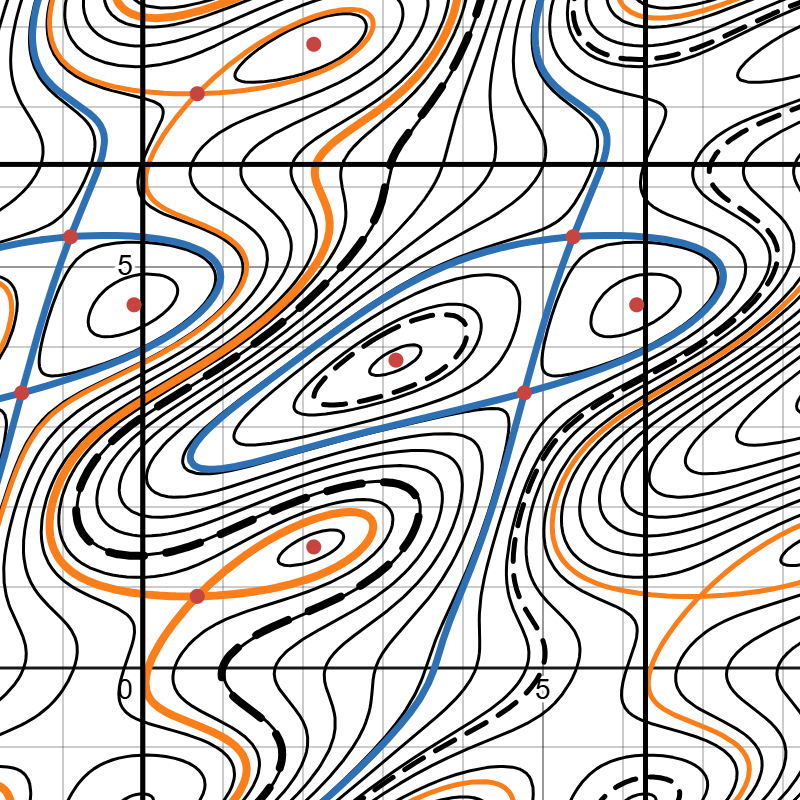}
\caption{Level curves of $u_{\phi, b,c,d} $ (phase curves of the vector fields $v_{\phi, b,c,d} $) for $\phi=1/3$, $c=1, b=2, d=1.0016$. Black square has sides of length $2\pi$.  Separatrices of the saddle $s_1$ are shown in orange thick; separatrices of  $s_2, s_3$ are shown in blue thick. Zeros of $v_{\phi, b,c,d} $ are shown in red. Black dashed lines are level curves $u_{\phi, b,c,d}=1,u_{\phi, b,c,d}=5$. }\label{fig-analytic}
\end{center}
\end{figure}
Clearly, the vector field $v_{\phi, b, c, d}$  is well-defined on the torus $T^2=(\bbR/2\pi \bbZ)^2$. It is easy to check that for $\phi\neq 0$,  it is transversal to $y=0$. Its phase curves are level curves of $u_{\phi, b, c, d}$. Since $u_{\phi, b, c, d}(x,0)=x$ and $u_{\phi, b, c, d}(x,2\pi) =x-2\pi \phi$, the Poincare map on $M_0$ is $x\to x+2\pi\phi$ whenever defined.  

Fig. \ref{fig-analytic} shows the level curves of $u_{\phi, b, c, d}$ for  $\phi=1/3, b=2,c=1,d\approx 1$.
The following statements on vector fields $v_{\phi, b, c, d}$ were verified numerically for $\phi=1/3, b=2$, $c \in [0.7, 1.1]$, $d \in [0.9, 1.3]$. 
\begin{itemize}
\item Vector fields $v_{\phi, b, c, d}$ have six zeros on the torus for all $c \in [0.7, 1.1]$, $d \in [0.9, 1.3]$.
\end{itemize}
This was checked by applying Python's fsolve method (that uses the Powell's hybrid method) with the dense mesh of parameter values and initial guesses. Zeros are marked on Fig. \ref{fig-analytic}. 
\begin{itemize}
\item  For all $c \in [0.7, 1.1]$, $d \in [0.9, 1.3]$, these six zeros remain hyperbolic; there are three saddle points, two minima of $u_{\phi, b, c, d}$ and one maximum of $u_{\phi, b, c, d}$.
\end{itemize}
This was checked by computing the Jacobians of the linear part of $v_{\phi, b, c, d}$. Jacobians remain greater than 2 in modulus. 

Let $s_1,s_2, s_3$ be the saddles of $v_{\phi, b, c, d}$, numbered left to right in $x$-coordinate, $x\in [0,2\pi]$.
\begin{itemize}
\item The meridian $y=0$, the level curve $\{u_{\phi, b, c, d}(x,y)=1, 0<y<1\}$, and a connected component of the level curve $\{u_{\phi, b, c, d}(x,y)=5, 0<y<1\}$ divide the torus into two domains $V_1$ and $V_2$; one of them ($V_1$) contains $s_1$ and a minimum of $u$, the other ($V_2$) contains $s_2,s_3$, and the remaining minimum and maximum of $u$. 
\end{itemize}
This was checked by (1) plotting these level curves for $u_{1/3, 2, 1, 1}$  (see Fig. \ref{fig-analytic}) and (2) verifying that for all $c \in [0.7, 1.1]$, $d \in [0.9, 1.3]$,  values of $u_{\phi, b, c, d}$ at its critical points in $0<y<1$ are not equal to $1 \pm 2\pi k ,5\pm 2\pi k$. This implies that critical points cannot move from one strip to another as parameters vary. 

Since level curves of $u_{\phi, b, c, d}$ are phase curves for $v_{\phi, b, c, d}$ and $u_{\phi, b, c, d}(x,0)=x$, we conclude that the correspondence map from $\{y=0\}$ to $\{y=1\}$ is defined near  $x=1,5$. Consider the domain $V_1$. Since $u_{\phi, b, c, d}$ is monotonic on $y=0$, out of four separatrices of $s_1$, only one can intersect $\{y=0\}$ and only one can intersect $\{y=1\}$. Thus the remaining separatrices of $s_1$ must form a separatrix loop in $V_1$ as shown in Fig. \ref{fig-analytic}, and  the correspondence map  from $\{y=0\}$ to $\{y=1\}$ in $V_1$ is well-defined (except a single point of intersection with a separatrix to which it extends continuously). 

Suppose that $u_{\phi, b, c, d}(s_2)=u_{\phi, b, c, d}(s_3)$. Since $u_{\phi, b, c, d}$ is monotonic on $y=0$ and $y=1$,  in the strip  $V_2$, only two of the eight separatrices of $s_2, s_3$ can intersect $\{y=0\} $ and $\{y=1\}$. The remaining separatrices must form three separatrix connections, as shown in Fig. \ref{fig-analytic}. The correspondence map  from $\{y=0\}$ to $\{y=1\}$ is everywhere defined, except two intersection points with separatrices to which it extends continuously. Thus it coincides with $x\to x+2\pi \phi$ as noted above.  

We claim that the condition  $u_{\phi, b, c, d}(s_2)=u_{\phi, b, c, d}(s_3)$ defines a graph of an analytic function $D=D(\phi, b, c)$ in the parameter space for $c \in [0.7, 1.1]$, $d\approx 2$, $\phi\approx 1/3$. This follows from the property below. 
 
\begin{itemize}
\item Derivative of  $u_{\phi, b, c, d}(s_3)-u_{\phi, b, c, d}(s_2)$ with respect to $d$ is positive   for all  $c \in [0.7, 1.1]$, $d \in [0.9, 1.3]$.
\end{itemize}

Derivative was computed via implicit function theorem. It remains between $2.05$ and $2.15$.

Hence for each $c \in [0.7, 1.1]$, the interval $d \in [0.9, 1.3]$ contains at most one value $d$ such that $u_{\phi, b, c, d}(s_2)=u_{\phi, b, c, d}(s_3)$. This value $d=D(\phi,b, c)$ was determined numerically for  $\phi=1/3, b=2$, $c \in [0.7, 1.1]$ and remains in $[0.9, 1.3]$. Since the condition is open, the same holds for all  $b\approx 2$, $\phi\approx 1/3$, and $D(\phi,b, c)$ is well-defined.  It is analytic since zeroes of $v_{\phi, b, c, d} $ depend analytically on parameters. 
 
For any vector field $v_{\phi, b, c, D(\phi,b, c)}$, let $U_1$ be the domain bounded by the separatrix connection of $s_1$; let $U_2$ be the domain bounded by the separatrix connections of $s_2, s_3$. Let stable separatrices of $s_1,s_2$ be $\gamma_1, \gamma_2$, let unstable separatrices of $s_1, s_3$ be  $\Gamma_1, \Gamma_2$.   Let $\rho(\phi, b, c) = u_{\phi, b, c, d}(s_2)-u_{\phi, b, c, d}(s_1)$ be the distance  between the intersection points of $\gamma_1$ and $\gamma_2$ with the meridian $M_0=\{y=0\}$. Numerically, we checked the following.

\begin{itemize}
\item The value $\rho(\phi, b, c)$ is not constant on the graph $(\phi, b, c, D(\phi, b, c))$: namely, for $c=0.7$ we have $d(c)\approx 1.23$ and $\rho\approx 3.40$ while for $c=1.1$ we get $d\approx 0.92$ and $\rho\approx 3.62$.
\end{itemize}

The function $\rho$ remains non-constant on $c$ for  $\phi\approx 1/3, d\approx 2$ since it is analytic. 

The remaining part of the proof is the same as for Theorem \ref{th-family}. 
While the vector field $v_{\phi, b, c, D(\phi, b, c)}$ with $\rho=\rho(\phi, b, c)$ is not orbitally topologically equivalent to $v_{\rho, \phi}$, the only difference is the explicit shape of the phase portrait inside $U_1, U_2$; the Poincare map and the behavior of separatrices $\gamma_1,\gamma_2,\Gamma_1, \Gamma_2$ is the same.  Therefore proofs of Lemma \ref{lem-eqvf} and Lemma \ref{lem-eqphi} apply to the family $v_{\phi, b, c, D(\phi, b, c)}$ with $\rho=\rho(\phi, b, c)$ without any modification, for any fixed $b$ close to $2$ and any fixed irrational $\phi$ close to $1/3$. These lemmas imply the second part of Theorem \ref{th-analytic}, which completes its proof.
\end{proof}

\bibliographystyle{amsalpha}
\bibliography{biblio}
\end{document}